\documentclass{amsart} 
\usepackage{amsthm, amsfonts, amssymb, amsmath}

\newtheorem{theorem}{Theorem}[section]
\newtheorem{lemma}[theorem]{Lemma}
\newtheorem{corollary}[theorem]{Corollary}

\newtheorem{remark}[theorem]{Remark}

\theoremstyle{definition}
\newtheorem{definition}[theorem]{Definition}
\newtheorem{example}[theorem]{Example}
\newtheorem{remark/example}[theorem]{Remark/Example}

 \let\oldlabel=\label
\def\prellabel{\marginparsep=1em\marginparwidth=44pt
   \def\label##1{\oldlabel{##1}\ifmmode\else\ifinner\else
    \marginpar{{\footnotesize\ \\ \tt
                 ##1}}\fi\fi}}

\numberwithin{equation}{section}

\def\AA{{\bf A }}
\def\CC{{\bf C }}

\newcommand{\Lt}{\operatorname{Lt}}
\newcommand{\Lc}{\operatorname{Lc}}

\newcommand{\GCD}{\operatorname{GCD}}

\newcommand{\Lex}{\operatorname{Lex}}
\newcommand{\chara}{\operatorname{char}}

\newcommand{\rank}{\operatorname{rank}}

\newcommand{\Hilb}{\operatorname{Hilb}}

\numberwithin{equation}{section}

\begin{document}

\title{Canonical Hilbert-Burch matrices for ideals of $k[x,y]$ }
\author{ Aldo Conca, Giuseppe Valla}
\address{Dipartimento di Matematica, Universit\'a di Genova, Via Dodecaneso 35, I-16146 Genova, Italy }
\email{conca@dima.unige.it, valla@dima.unige.it }
\date{}
 
\begin{abstract} 
An Artinian ideal $I$ of $k[x,y]$ has many Hilbert-Burch matrices. 
We show that there is a canonical choice. As an application, we determine the dimension of certain affine Gr\"obner cells and their Betti strata recovering results of Ellingsrud and Str\o mme, G\"ottsche and Iarrobino. 
\end{abstract}
 
\maketitle
 \section{Introduction} 
 Let $k$ be a field. Let $R$ be the polynomial ring $k[x_1,\dots,x_n]$ and $\tau$ be a term order on $R$.
Given a non-zero $f\in R$ we denote by $\Lt_\tau(f)$ the largest term with respect to $\tau$ appearing in $f$. For an ideal $I$ of $R$ we denote by $\Lt_\tau(I)$ the (monomial) ideal generated by $\Lt_\tau(f)$ with $f\in I\setminus\{0\}$. Let $E$ be a monomial ideal of $R$. 
Consider the set $V(E)$ of the homogeneous ideals $I$ of $R$ such that $\Lt_\tau(I)=E$. The set $V(E)$ has a natural structure of affine variety. Namely, given $I$ in $V(E)$, we can consider $I$ as a point in an affine space $\AA^N$ with coordinates given by the coefficients of the non-leading terms in the reduced Gr\"obner basis of $I$, see Section \ref{param} for details. The equations defining (at least set-theoretically) $V(E)$ can be obtained form Buchberger's Gr\"obner basis criterion. Provided $\dim_k R/E$ is finite, one can give the structure of affine variety also to the set $V_0(E)$ of the ideals $I$ (homogeneous or not) such that $\Lt_\tau(I)=E$. 
 
These varieties play important roles in many contexts such as, for example, the study of various types of Hilbert schemes and the problem of deforming non-radical to radical or prime ideals, see \cite{AS, Br,CRV, ES1,ES2,Go1,Go2,I1,I2,IY, MS}. 

Many of the equations defining $V(E)$ or $V_0(E)$ contain parameters that appear in degree $1$ and that can be eliminated. It happens quite often that, after getting rid of the superfluous parameters, one is left with no equations, that is, the variety is an affine space. But, in general, $V(E)$ can be reducible and it can have irreducible components that are not affine spaces, see the examples \ref{exe1},\ref{exe2} and \ref{exe3}.
 
On the other hand, for $n=2$ and $d=\dim_k R/E<\infty$, it is known that $V_0(E)$ and $V(E)$ are affine spaces. This is a consequence of general results of Bialynicki-Birula \cite{BB1,BB2} on smooth varieties with $k^*$-actions. Here it is important to note that $V_0(E)$ coincides with the set of points of the Hilbert scheme $\Hilb^d(\AA^2)$ that degenerate to $E$ under a suitable $k^*$-action associated to a weight vector representing the term order on monomials of degree $\leq d+1$. By the analogy with Schubert cells for Grassmannians, we name $V_0(E)$ and $V(E)$ Gr\"obner cells. 

Our goal here is to show that for $n=2$ and $\tau$ the lexicographic order induced by $x>y,$ both $V(E)$ and $V_0(E)$ can be described as affine spaces in a very explicit way, see \ref{realmain}. To achieve this goal we identify canonical Hilbert-Burch matrices of the ideals involved. The main point is to introduce (redundant) systems of generators for the ideals in $V_0(E)$, that, instead of being themselves ``simple", have ``simple" syzygies.

We can then easily deduce formulas for the dimensions of $V(E)$ and $V_0(E)$ and of two other subvarieties of $V_0(E)$, see \ref{main}. Dimension formulas for these varieties were originally obtained in \cite{ES1,ES2, I1, Go2, IY}. 
In Section \ref{betti} we reprove and generalize some results of Iarrobino \cite{I2} concerning the Betti strata of $V(E)$. 

For standard facts on Gr\"obner bases we refer the reader to \cite{KR} or \cite{E}. The results of this paper were discovered, suggested and double-checked by extensive computer algebra experiments performed with CoCoA \cite{Co}.

\section{$V(E)$ as an affine variety} 
\label{param}

With the notations introduced above, we first recall how $V(E)$ and $V_0(E)$ can be given the structure of affine varieties. For every minimal monomial generator $m$ of $E$ consider the polynomial 
$$f_m=m-\sum \lambda(m,m') m'$$
 where the sum is extended the monomials $m'\not\in E$ such that $\deg m=\deg m'$ and $m'<m$ with respect to $\tau$. 
 Denote by $N$ the total number of the parameters $\lambda(m,m')$. The property of being a Gr\"obner basis for the $f_m$'s is turned into the vanishing of polynomials, say $B_1,\dots,B_r$, on the parameters $\lambda(m,m')$. Since an ideal has a unique reduced Gr\"obner basis, the points of the affine variety of $\AA^N$ defined by the vanishing of the $B_i$ are in bijection with the elements of $V(E)$. The polynomials $B_i$ can be explicitly computed through the Buchberger's criterion for Gr\"obner basis. There are many degrees of freedom in the application of the Buchberger's criterion (e.g. one can use all the S-pairs or carefully choosen subsets of them, the reduction process can be performed in various ways, and so on). So the actual nature of the polynomials $B_i$ depend on these choices but, of course, not the variety that they define. 
 
Similarly, if $\dim_k R/E$ is finite, one can give the structure of affine variety to $V_0(E)$ by dropping the assumption that $\deg m=\deg m'$ in the definition of $f_m$. 

As said in the Introduction, the varieties $V(E)$ and $V_0(E)$ quite often are affine spaces. Roughly speaking, what happens is the following. Say $m,n$ are monomial generators of $E$, $m'<m$ and $t=m'n/\GCD(m,n)$ satisfies $t\not\in E$. Then the coefficient of $t$ 
in the $S$-polynomial associated to $f_m$ and $f_n$ is just $\lambda(m,m')$ or $\lambda(m,m')-\lambda(n,n')$ depending on whether 
there exists $n'<n$ such that $t=n'm/\GCD(m,n)$. Performing the reduction procedure, $\lambda(m,m')$ cannot be cancelled because at each iteration the degree of the coefficients involved increases by $1$. At the end of the reduction procedure, the coefficient of $t$ in the polynomial we are left with must vanish. Therefore we have equations of the form: 
\begin{equation} 
\label{elimina}
\lambda(m,m')+B=0 \quad \mbox{ or } \quad \lambda(m,m')-\lambda(n,n')+B=0 
\end{equation} 
where $B$ is a polynomial in the $\lambda(*,*)$ not involving monomials of degree $1$. 
Of course if $B$ does not involve $\lambda(m,m')$ at all then we can use \ref{elimina} to get rid of the parameter $\lambda(m,m')$ from the equations. This elimination process can be iterated. In many cases, at the end of the elimination process, the equations vanish completely, and this shows that the associated variety is an affine space. We have implemented this rough algorithm in CoCoA \cite{Co}. We have tested, for instance, that for $\tau=\Lex$, $n=3$ and $E$ any ideal generated by monomials of degree $3$ then $V(E)$ is an affine space. 

The following examples show that in general the variety $V(E)$ has a more complicated structure. For simplicity, the coordinates of the ambient affine spaces $\lambda(m,m')$ are denoted by $a_i$. 

 \begin{example}\label{exe1} 
 Set $n=3$, $E=(x_3^4, x_2^4, x_1x_2^2x_3, x_1^3x_3)$ and $\tau=\Lex$. Then $V(E)$ is a subvariety of $\AA^{17}$, the inclusion being given by the parametrization:

\begin{equation*}
\begin{array}{rl}
x_1^3x_3&-x_1^2x_2^2a_1-x_1^2x_2x_3a_2-x_1^2x_3^2a_3-x_1x_2^3a_4-x_1x_2x_3^2a_5-x_1x_3^3a_6\\
&-x_2^3x_3a_7-x_2^2x_3^2a_8-x_2x_3^3a_9,\\ 
x_1x_2^2x_3&-x_1x_2x_3^2a_{10}-x_1x_3^3a_{11}-x_2^3x_3a_{12}-x_2^2x_3^2a_{13}-x_2x_3^3a_{14},\\
x_2^4 &-x_2^3x_3a_{15}-x_2^2x_3^2a_{16}-x_2x_3^3a_{17},\\
x_3^4.& 
\end{array}
\end{equation*} 
Buchberger's criterion gives $3$ equations. Two of them can be written as: 
 \begin{equation*} 
 \begin{array}{ll}
 a_{14}=-a_{10}^2a_{12}-a_{10}a_{13}-a_{11}a_{12} ,\\
 a_9= 2 a_1a_{10}^2a_{12}^2a_{15}+ \mbox{ other $46$ terms in the $a_i$'s not involving $a_9$ and $a_{14}$}.
\end{array}
\end{equation*} 
Setting $b_{17}= a_{10}^3-a_{10}^2a_{15} + 2 a_{10}a_{11}-a_{10}a_{16}-a_{11}a_{15}-a_{17}$, the third equation is 
$a_1b_{17} =0.$
Hence $V(E)$ has two irreducible components both isomorphic to $\AA^{14}$. 
\end{example} 

 \begin{example} \label{exe2} 
Let $n=4$, $E=(x_{4}^2, x_{2}x_{4}, x_{2}^2, x_{1}x_{4})$ and $\tau=\Lex$. Then $V(E)$ is a subvariety of $\AA^{8}$, the inclusion being given by the parametrization: 
\begin{equation*}
\begin{array}{rl}
x_{4}^2,&\\
x_{2}x_{4} & - x_{3}^2a_{1} - x_{3}x_{4}a_{2} ,\\ 
x_{2}^2 &- x_{2}x_{3}a_{3} - x_{3}^2a_{4} - x_{3}x_{4}a_{5},\\ 
x_{1}x_{4} & - x_{2}x_{3}a_{6} - x_{3}^2a_{7} - x_{3}x_{4}a_{8}.
\end{array}
\end{equation*}
The parameters $a_{1}, a_{7}, a_{4}$ can be eliminated, so that $V(E)$ is indeed contained in $\AA^5$.
 After renaming $b_{3}=2a_{2} - a_{3}$ the defining ideal of $V(E)$ in $\AA^5$ takes the form $b_{3}a_{6}, a_{5}a_{6}$. Hence $V(E)$ has two irreducible components, one isomorphic to $\AA^3$ and the other to $\AA^4$. 
 \end{example}

 \begin{example} \label{exe3} 
Let $n=4$, $E=(x_{4}^2, x_{2}x_{4}, x_{1}x_{4}, x_{1}x_{2}, x_{1}^2)$ and $\tau=\Lex$.
Then $V(E)$ is a subvariety of $\AA^{16}$, the inclusion being given by the parametrization: 
$$
\begin{array}{rl}
x_{4}^2, & \\
x_{2}x_{4} & - x_{3}^2a_{1} - x_{3}x_{4}a_{2},\\ 
x_{1}x_{4} & - x_{2}^2a_{3} - x_{2}x_{3}a_{4} - x_{3}^2a_{5} - x_{3}x_{4}a_{6},\\
x_{1}x_{2} & - x_{1}x_{3}a_{7} - x_{2}^2a_{8} - x_{2}x_{3}a_{9} - x_{3}^2a_{10} - x_{3}x_{4}a_{11},\\
x_{1}^2 & - x_{1}x_{3}a_{12} - x_{2}^2a_{13} - x_{2}x_{3}a_{14} - x_{3}^2a_{15} - x_{3}x_{4}a_{16}.
\end{array} 
$$
The parameters 
$$a_{1}, a_{3}, a_{4}, a_{5}, a_{13}, a_{10}, a_{14}, a_{15}$$ 
can be eliminated, so that $V(E)$ is indeed contained in $\AA^8$. After renaming 
$$b_{9}=a_{2}a_{8} + a_{7}a_{8} - a_{6} + a_{9}, \quad b_{12}=2a_{6} + b_{9} - a_{12}, \quad b_{7}=a_{2} - a_{7}$$ the defining ideal of $V(E)$ in $\AA^8$ takes the form 
$$(b_{7}b_{9},b_{9}b_{12}, a_{11}b_{12} - b_{7}a_{16}).$$ Hence $V(E)$ has two components, one is isomorphic to $\AA^6$ and the other is a quadric hypersurface of rank $4$ in $\AA^7$. 
\end{example} 

 \section{Ideals in $k[x,y]$} 
\label{kxy}

Form now on, let $k$ be a field, $R=k[x,y]$ be the polynomial ring over $k$. We equip $R$ with the lexicographic term order $>$ induced by $x>y$. 

Given a monomial ideal $E\subset R$ with $\dim_k R/E<\infty$ we want to describe the set of ideals: 

$$V_0(E)=\{ I \mbox{ such that } \Lt(I)=E\}$$ and its subsets
$$V_1(E)=\{ I \mbox{ such that } \Lt(I)=E \mbox{ and } y\in \sqrt{I})\},$$ 
$$V_2(E)=\{ I \mbox{ such that } \Lt(I)=E \mbox{ and } \sqrt{I}=(x,y)\},$$
$$V(E)=V_3(E)=\{ I \mbox{ such that } \Lt(I)=E \mbox{ and } I \mbox{ is homogeneous} \}.$$

Our goal is to prove Theorem \ref{realmain}. As a corollary we have: 

\begin{corollary} \label{main}
The set $V_0(E)$ is an affine space. The subsets $V_1(E), V_2(E)$ and $V_3(E)$ are also affine spaces, indeed coordinate subspaces of $V_0(E)$. Furthermore

$$\dim V_i(E)=\left\{
\begin{array}{ll}
\dim_k R/E+\min\{ j : y^j\in E\} & \mbox{ if } i=0,\\
\dim_k R/E  & \mbox{ if } i=1,\\
\dim_k R/E-\min\{ j : x^j\in E\} & \mbox{ if } i=2,\\
\# \mathcal{S}(E) & \mbox{ if } i=3.
\end{array}
\right.
$$
where $\mathcal{S}(E)$ is a set described below and $\# \mathcal{S}(E)$ denotes its cardinality. 
\end{corollary} 

The dimension formulas for $V_0(E),V_1(E)$ and $V_2(E)$ have been proved originally in \cite{ES1,ES2}. 
A dimension formula for $V_3(E)$ appears in \cite{I1,IY} for lex-segments $E$ and in \cite{Go1} for general $E$ . 

To prove \ref{main} one could try no analyze the equations coming from Buchberger's criterion. But this turns out to be quite difficult. Instead we parametrize the syzygies and identify canonical Hilbert-Burch matrices. 

We introduce a piece of notation. Given a monomial ideal $E$ such that $\dim_kR/E$ is finite, we set $t=min\{j:x^j\in E\}$, $m_0=0,$ and for every $1\le i\le t$ $m_i=min\{j:x^{t-i}y^j\in E\}$. It is clear that
$m_0=0<m_1\leq m_2\leq \dots \leq m_t$, 
$$E=(x^t,x^{t-1}y^{m_1},\dots,xy^{m_{t-1}},y^{m_t})$$ and 
 $\dim_k R/E=\sum_{i=0}^t m_i$. These generators of $E$ are not minimal in general. They minimally generate $E$ if and only if $m_0<m_1<m_2<\dots <m_t$, that is, $E$ is a lex-segment ideal.
By construction, the correspondence 
$$E \leftrightarrow (m_0,\dots, m_t)$$
 is a bijection between monomial ideals of $R$ with radical equal to $(x,y)$ and sequences of integers $0=m_0<m_1\leq m_2\leq \dots \leq m_t$. 
 
Given $E$ or, equivalently $(m_0,\dots, m_t)$, we set 
$$d_i=m_i-m_{i-1}$$
 for $i=1,\dots,t$. Here $d_1>0$ and $d_i\geq 0$ for every $i=2,\dots,t$. Clearly, $E$ can be as well described via the vector $(d_1,\dots,d_t)$. Furthermore, the lex-segment correspond exactly to the vectors with $d_i>0$ for $i=1,\dots,t$. 

The matrix 
$$M_0(E)=\begin{pmatrix}
y^{d_1}&0 & 0& \cdots & 0& 0\\
-x& y^{d_2}& 0& \cdots&0 & 0 \\
0& -x& y^{d_3}& \cdots &0&0 \\ 
\vdots&\vdots & \ddots& \ddots& \vdots& \vdots\\
 0&0 &0 & -x&y^{d_{t-1}} \\
 0& 0&0 & 0&-x&y^{d_t}\\
 0& 0&0 &0 &0 &-x
 \end{pmatrix}$$
 has size $(t+1)\times t$ and is a Hilbert-Burch matrix of $E$ in the sense that the (signed) $t$-minors of $M_0(E)$ are the monomials $x^{t-i}y^{m_i}$ and the columns generate their syzygy module. 
 
The matrix $M_0(E)$ represents a map from $F_1=\oplus_{i=1}^t R(-t+i-m_i-1)$ to $F_0=\oplus_{i=1}^{t+1} R(-t+i-1-m_{i-1})$. It is useful to consider also the corresponding degree matrix $U(E)=(u_{ij})$. The entries of $U(E)$ are the degrees of the (homogeneous) entries of every matrix representing a map of degree $0$ from $F_1$ to $F_0$. We have 
 
 \begin{equation}\label{uij}
 u_{ij}=m_j-m_{i-1}+i-j \mbox{ for } i=1,\dots,t+1 \mbox{ and } j=1,\dots,t 
 \end{equation}
 Notice that $u_{ii}=m_i-m_{i-1}=d_i$ and $u_{i+1,i}=1$ for every $i=1,\dots,t$. Define: 
 
 $$\mathcal{S}(E)=\{ (i,j) : 1\leq j<i\leq t+1 \mbox{ and } 0\leq u_{ij}<d_j\}$$
 
 \begin{definition}
 \label{defi} 
Let $T_0(E)$ be the set of $(t+1)\times t$ matrices $N=(n_{i,j})$ where $$
 n_{i,j}=\left\{ \begin{array}{ll}
 0 & \mbox{ if } i<j \\
\mbox{a polynomial in k[y] of degree } <d_j & \mbox{ if } i\geq j
\end{array} 
\right.
$$
Further consider the following conditions: 
\begin{itemize} 
\item[(1)] $n_{i,i}=0$ for every $i=1,\dots,t$. 
\item[(2)] For every $j$ such that $d_j>0$ the polynomial $n_{i,j}$ has no constant term for every $i=j+1,\dots,k+1$ where $k=\max\{ v: j\leq v\leq t \mbox{ and } m_v=m_j \}$. 
\item[(3)] $$
 n_{i,j}=\left\{ \begin{array}{ll}
 0 & \mbox{ if } (i,j)\notin \mathcal{S}(E)\\
p_{ij}y^{u_{ij}} &\mbox{ if } (i,j)\in \mathcal{S}(E)
\end{array} 
\right.
$$ 
 \end{itemize} with $p_{ij}\in k.$ 
 Accordingly we define 
$$
\begin{array}{ll} 
T_1(E)=\{ N \in T_0(E) : N \mbox{ satisfies } (1) \} \\ \\ 
T_2(E)=\{ N \in T_0(E) : N \mbox{ satisfies } (1) \mbox{ and } (2) \} \\ \\ 
T_3(E)=\{ N \in T_0(E) : N \mbox{ satisfies } (3) \} 
\end{array} 
$$
\end{definition}

\begin{theorem}
\label{realmain}
 For every monomial ideal $E$ the map $\phi: T_0(E) \to V_0(E)$ defined by sending $N\in T_0(E)$ to the ideal of $t$-minors of the matrix $M_0(E)+N$ is a bijection. Furthermore, the restriction of $\phi$ induces bijections between $T_i(E)$ and $V_i(E)$ for $i=1,2,3$. 
\end{theorem} 

By construction, the sets $T_i(E)$ are affine spaces and their dimension can be easily computed from their defining conditions. Therefore Theorem \ref{main} is an immediate consequence of \ref{realmain}. 

Before embarking in the proof of \ref{realmain} let us consider one example. 
\begin{example} Let $E=(x^3,xy^3,y^5)=(x^3,x^2y^3,xy^3,y^5)$. Then $m=(0,3,3,5)$, $d=(3,0,2)$ and 
$$M_0(E)=
\begin{pmatrix}
 y^3 & 0 & 0 \\
 -x & 1& 0 \\
 0 & -x & y^2 \\
 0 & 0 & -x \end{pmatrix}
\ \ \ \ U(E)=\begin{pmatrix}
 3 & 2 & 3 \\
 1 & 0& 1 \\
 2 & 1 & 2 \\
 1 & 0 & 1
\end{pmatrix}$$ 
We have $t=3$, $\min\{i:y^i\in E\}=5,$ $\dim R/E =11$ and $\# \mathcal{S}(E)=4.$

The matrices in $T_0(E)$ have the form: 

$$\begin{pmatrix}
 n_{1,1} & 0 & 0 \\
 n_{2,1} & 0& 0 \\
 n_{3,1} & 0& n_{3,3} \\
 n_{4,1} & 0 & n_{4,3} 
 \end{pmatrix}
 $$
 where the $n_{i,1}$ are polynomials in $y$ of degree $<3$ and the $n_{i,3}$ are polynomials in $y$ of degree $<2$. 
 The matrices in $T_1(E)$ are those of $T_0(E)$ such that $n_{1,1}=0$ and $n_{3,3}=0$. The matrices in $T_2(E)$ are those of $T_1(E)$ such that 
 $n_{2,1},n_{3,1},n_{4,3}$ have no constant term. Finally the matrices in $T_3(E)$ have the form: 
 
 $$\begin{pmatrix}
 0 & 0 & 0 \\
 p_{21}y & 0 & 0 \\
 p_{31} y^2 & 0 & 0 \\
 p_{41}y & 0 & p_{43}y 
 \end{pmatrix}
 $$
 where the $p_{ij}\in k$. 
 
 As predicted by \ref{main} we get 
 $\dim T_0(E)=16$, $\dim T_1(E)=11$, $\dim T_2(E)=8,$ and $\dim T_3(E)=4$.
 \end{example}

 The proof of \ref{realmain} consists of the following steps: 
 
 \begin{itemize} 
\item[(Step 1)] The map $\phi$ is well-defined.
\item[(Step 2)] The map $\phi$ is bijective. 
\item[(Step 3)] For $i=1,2,3$ we have $\phi(N)\in V_i(E)$ iff $N\in T_i(E)$. 
\end{itemize} 

Let us begin with

\begin{proof}[ Proof of Step 1] 
 For $N\in T_0(E)$ set $I=\phi(N)$. We show that $\Lt(I)=E$. For $i=0,\dots,t$ let $f_i$ be $(-1)^{t-i}$ times the determinant of the submatrix of $M_0(E)+N$ obtained by deleting the $(i+1)$-th row. By construction, $\Lt(f_i)=x^{t-i}y^{m_i}$ and $\Lc(f_i)=1$. We show that $f_0,\dots,f_t$ form a Gr\"obner basis of $I$. The syzygy module of leading terms of the $f_i$ is generated by the syzygies 
 \begin{equation}
 y^{d_{i}}(x^{t-i+1}y^{m_{i-1}})- x (x^{t-i}y^{m_{i}}) =0
 \label{syzygy}
 \end{equation}
 with $i=1,\dots,t$. To prove that the $f_i$'s form a Gr\"obner basis it is enough to show that the S-polynomials associated to these syzygies reduce to $0$. Since we have 
$$y^{d_i}f_{i-1}-xf_i+\sum_{j=i-1}^t n_{j+1,i} f_j=0$$ 
it is enough to show that if $y^{d_i}f_{i-1}-xf_i\neq 0$ then 
$\Lt( n_{j+1,i} f_j)\leq \Lt(y^{d_i}f_{i-1}-xf_i)$
for every $n_{j+1,i}\neq 0$. 
Note that the non-zero factors $n_{j+1,i}f_j$ have leading terms involving different powers of $x$. Hence $\max( \Lt( n_{j+1,i} f_j) : n_{j+1,i}\neq 0)=\Lt(y^{d_i}f_{i-1}-xf_i)$. 
\end{proof} 

Step 2 will be a corollary of the following two lemmas: 

\begin{lemma}\label{basic1} Let $I$ be an ideal of $R$ such that $\Lt(I)=E$ and let $f_0,\dots, f_t\in I$ such that $\Lt(f_i)=x^{t-i}y^{m_i}$ and $\Lc(f_i)=1$. Then for every $f\in I$ such that $\Lt(f)=x^{t-i}y^b$ for some $0\leq i \leq t$ there exist
polynomials $g_j \in k[y]$ wih $j=i,\dots, t$ with $\deg g_i=b-m_i$ such that $f+g_if_i+\dots+g_tf_t=0$. 
\end{lemma} 

\begin{proof} By assumption, $f_0,\dots,f_t$ is a Gr\"obner basis of $I$. Hence $x^{t-i}y^b$ is divisible by some $x^{t-j}y^{m_j}$. Hence $t-j\leq t-i$ and $m_j\leq b$. It follows that $i\leq j$ and $m_i\leq m_j\leq b$. Therefore $f-\Lc(f)y^{b-m_i}f_i$ is still in $I$ and has a smaller leading term (if it is non-zero). We get the desired representation by iterating the procedure. 
\end{proof}

\begin{lemma}\label{basic2} 
 Let $I$ be an ideal of $R$ such that $\Lt(I)=E$. Then there exist $f_0,\dots, f_t\in I$ such that: 
\begin{itemize}
\item[(1)] $\Lt(f_i)=x^{t-i}y^{m_i}$ and $\Lc(f_i)=1$ for every $i=0,\dots, t$. 
\item[(2)] For every $i=1,\dots t$ there exists $n_{j+1,i}\in k[y]$ with $i-1\leq j\leq t$ and $\deg n_{j+1,i}<d_i$ such that 
\begin{equation}
y^{d_i}f_{i-1}-xf_i+\sum_{j=i-1}^t n_{j+1,i} f_j=0
\label{super1}
\end{equation}
\end{itemize} 
Furthermore the polynomials $f_i$ and $n_{j+1,i}$ with these properties are uniquely determined by $I$. 
\end{lemma} 

\begin{proof} We prove the existence first. A set of polynomials $f_0,\dots, f_t\in I$ satisfying $(1)$ clearly exists. We show how to modify them it in order to fulfill (2). For a given $k$, $1\leq k\leq t$ suppose that we have already modified $f_k,\dots,f_t$ so that (1) is still fulfilled and that (2) is fulfilled for $i=k+1,\dots,t$. We show how to modify $f_{k-1}$ in order to fulfill (2) for $i=k$. Note that $y^{d_k}f_{k-1}-xf_k$ is in $I$ and involves only terms with $x$-exponent $\leq t-(k-1)$ and that if $x^{t-(k-1)}y^b$ is indeed present then $b<m_k$. By \ref{basic1} we have that there exists $g_{k-1},\dots,g_t\in k[y]$ such that $g_{k-1}$ is either $0$ or of degree $<d_k$ and 
 \begin{equation}
 \label {prima} 
 y^{d_k}f_{k-1}-xf_k+g_{k-1}f_{k-1}+g_{k}f_{k}+\dots+g_tf_t=0 
 \end{equation} 
 
Set $h=y^{d_k}+g_{k-1}$ and perform for $j=k,\dots,t$ division with remainder: $g_j=hq_j+r_j$ with $q_j,r_j\in k[y]$ and $r_j$ either $0$ or of degree $<d_k$. Then we have 

\begin{equation}
\label {seconda} 
 y^{d_k}f'_{k-1}-xf_k+g_{k-1}f'_{k-1}+r_{k}f_{k}+\dots+r_tf_t=0
 \end{equation} 
 
with $f'_{k-1}=f_{k-1}+q_{k}f_k+\dots+q_tf_t$. Note that $f'_{k-1}$ is in $I$ and 
$\Lt(f_{k-1}')=\Lt(f_{k-1})$ and $\Lc(f_{k-1})=\Lc(f'_{k-1})$. We may replace $f_{i-1}$ with $f'_{i-1}$ and \ref{seconda} is the desired relation. 

We prove now the uniqueness of the $f_i$'s and $n_{j+1,i}$ fulfilling (1) and (2). Suppose we have other polynomials $f'_i$ and $n'_{j+1,i}$'s fulfilling (1) and (2). Note that $f_t=f_t'$ since they are both the monic generator of $I\cap k[y]$. So we may assume that $f_j=f_j'$ for $j=k,\dots,t$ and show that $f_{k-1}=f_{k-1}'$. 
By assumption we have equations: 

\begin{equation}
y^{d_k}f_{k-1}-xf_k+ n_{k,k} f_{k-1}+ \sum_{j=k}^t n_{j+1,k} f_j =0
\label{uniq}
\end{equation}

\begin{equation}
y^{d_k}f'_{k-1}-xf'_k+ n'_{k,k} f'_{k-1}+ \sum_{j=k}^t n'_{j+1,k} f'_j =0
\label{uniq1}
\end{equation}

where the $n_{j+1,k}$ and $n'_{j+1,k}$ are polynomials in $k[y]$ of degree $<d_k$. 

By \ref{basic1} applied with $f=f'_{k-1}$ we have an equation: 

\begin{equation}
f'_{k-1}=f_{k-1}+g_kf_k+\dots+g_tf_t
\label{uniq2}
\end{equation}

with $g_j\in k[y]$. Set $h=y^{d_i}+ n_{k,k}$ and $h'=y^{d_i}+ n'_{k,k}$. Replacing $f'_{k-1}$ in \ref{uniq1} with the right hand side of \ref{uniq2} and then subtratting \ref{uniq} we obtain: 

\begin{equation}
(h'-h)f_{k-1}+\sum_{j=k}^t (h'g_j+n'_{k,j+1} -n_{k,j+1}) f_j=0
\label{uniq3}
\end{equation}

Since the leading terms of the $f_i$'s involves distinct powers of $x$, the $f_i$'s are linearly independent over $k[y]$. Hence the coefficients $h'g_j+n'_{k,j+1} -n_{k,j+1}$ of \ref{uniq3} must be $0$. Therefore $h'g_j=-n'_{k,j+1}+n_{k,j+1}$. But $n'_{k,j+1}+n_{k,j+1}$ has degree $<d_k$ and $h'$ has degree $d_k$. Therefore $g_j=0$ for every $j$ and hence $f_{k-1}=f'_{k-1}$. Having shown that the $f_i$'s fulfilling (1) and (2) are uniquely determined by $I$, it remains that the coefficients $n_{j+1,i}$ are also uniquely determined. This is easy: given others coefficients $n'_{j+1,i}$ satisfying (2), say
\begin{equation}
y^{d_i}f_{i-1}-xf_i+\sum_{j=i-1}^t n'_{j+1,i} f_j=0
\label{super2}
\end{equation}
we may subtract \ref{super1} from \ref{super2} and get 

$$\sum_{j=i-1}^t (n'_{j+1,i} -n_{j+1,i})f_j=0.$$

This implies $n'_{j+1,i}=n_{j+1,i}$ by the linear indipendence of the $f_i$'s over $k[y]$. 
\end{proof} 

We are ready to prove: 

\begin{proof}[Proof of Step 2]
We first prove that $\phi$ is injective. Suppose $I=\phi(N)=\phi(N')$ for matrices $N,N'\in T_0(E)$. We have seen in the proof of Step 1 that the signed $t$-minors $f_0,\dots,f_t$ of $M_0(E)+N$ fulfill (1) and (2) of \ref{basic2}. The same it is true for the signed $t$-minors $f'_0,\dots,f'_t$ of $M_0(E)+N'$. By the uniqueness of the $f_i$'s in \ref{basic2} we have that $f_i=f_i'$ for every $i$. By the uniqueness of the coefficients of the equation of \ref{basic2} the conclude that $N=N'$. 

We show now that $\phi$ is surjective. Let $I\in V_0(E)$. We may find $f_0,\dots,f_t\in I$ satisfying (1) and (2) of \ref{basic2}. The equation \ref{super1} is the reduction to $0$ of the $S$-polynomial corresponding to the syzygy \ref{syzygy} among the leading terms. As we know that these syzygies generate the syzygy module of the leading term of the $f_i$, Schreyer's theorem \ implies that the equations \ref{super1} give a system of generators for the syzygy module of the $f_i$'s. The corresponding matrix is of the form $M_0(E)+N$ with $N\in T_0(E)$ and the Hilbert-Burch theorem implies that $\phi(N)=I$. 
\end{proof} 

Now we prove: 

\begin{proof}[Proof of Step 3] 
Throughout the proof, $N$ denotes a matrix in $T_0(E)$, $I=\phi(N)$ and $f_0,\dots,f_t$ the signed $t$-minors of $M_0(E)+N$.

Since the $f_i's$ form a Gr\"obner basis with respect to the lexicographic order, then $f_t=\Pi_{i=1}^t (y^{d_i}+n_{i,i})$ generates $I\cap k[y]$. 
We have that $y\in \sqrt{I}$ iff $f_t$ divides some power of $y$. But this is clearly equivalent to the vanishing of $n_{i,i}$ for $i=1,\dots,t$. This proves that $N\in T_1(E)$ iff $\phi(N)\in V_1(E)$. 

To prove that $N\in T_2(E)$ iff $\phi(N)\in V_2(E)$ we may assume that $N\in T_1(E)$ and we show that 
$\sqrt{I}=(x,y)$ iff the $N$ fulfills condition (2) of \ref{defi}. As we know already that $y\in \sqrt{I}$, we have that $ \sqrt{I}=\sqrt{I+(y)}$. Replace $y$ with $0$ is $M_0(E)+N$, and call $W_1$ the resulting matrix. The first row of $W_1$ is $0$ (since $d_1>0$). Denote by $W$ the submatrix of $W_1$ obtained by deleting the first row. By construction $I+(y)=(\det W, y)$. We have to show that $\det W$ is a power of $x$ iff $N$ fulfills condition (2) of \ref{defi}. 
 
Let $C=\{ i : i=1,\dots,t \mbox{ and } d_i>0\}$, say $C=\{i_1,\dots, i_p\}$ with $i_1<\dots<i_p$. By assumption, $i_1=1$ and we set $i_{p+1}=t+1$ by convention. 
The matrix $W$ has a block decomposition 

$$W=\begin{pmatrix}
J_1&0 & 0& \cdots & 0\\
 *&J_2& 0& \cdots & 0 \\
*& *& J_3& \cdots &0 \\ 
\vdots&\vdots & \ddots & \ddots& \vdots\\
 *& *& \cdots &*&J_p\\ 
 \end{pmatrix}$$
 where each $J_v$ is a square block of size, say, $u=i_{v+1}-i_v$ and has the form 

$$\begin{pmatrix}
-x+a_1& 1 & 0& \cdots& \cdots & 0\\
 a_2 & -x& 1& 0 & \cdots & 0 \\
a_3 & 0& -x& \cdots & \cdots &0 \\ 
\vdots&\vdots & \vdots & \ddots & \ddots& \vdots\\
 & 0& \cdots & 0 &-x&1\\ 
a_u & 0& \cdots & 0 &0&-x\\ 
 \end{pmatrix}$$
 where $a_j=n_{i_v+j,i_v}(0)$ for $j=1,\dots,u$. Now $\det W=\Pi_v \det J_v$. The determinant of the matrix $J_v$ is, up to sign, 
$x^u-a_1x^{u-1}-a_2x^{u-2}-\dots-a_u$. Hence $\det W$ is a power of $x$ if and only if the coefficients $a_j$ in every the $J_v$ are $0$. This is condition (2) of \ref{defi}.

Finally we have to show that $N\in T_3(E)$ iff $I$ is homogeneous. The ``only if" direction is an immediate consequence of the fact that the matrix $M_0(E)+N$ is homogeneous for every $N\in T_3(N)$. The ``if" direction follows from the observation that the polynomials $f_i$ and $n_{i,j}$ of \ref{basic2} are homogeneous if we start with a homogeneous ideal $I$. 
\end{proof} 

 \section{Betti strata of $V(E)$} 
 \label{betti} 
 
Let $h=h(z)$ be the Hilbert series of a graded Artinian quotient of $R$. It is known that $h(z)$ is of the form $$h(z)=1+2z+\dots+cz^{c-1}+\sum _{j=c}^sh_jz^j$$ with $s+1\ge c \ge h_c\ge \dots \ge h_s>0$. Denote by $\mathbb{G}(h)$ the variety that parametrizes graded ideals $I$ in $R$ such that the Hilbert series $h_{R/I}(z)=h(z).$ Iarrobino proved in \cite {I1} that $\mathbb{G}(h)$ is a smooth projective variety whose dimension is given by the beautiful formula: 
\begin{equation}
\dim \mathbb{G}(h)=h_c+\sum_{j=c}^s p_jp_{j+1}
\label{bella}
\end{equation}
where $p(z)=\sum_0^{s+1}p_iz^i=(1-z)h(z)$ is the first difference of $h(z)$.

Among the ideals with Hilbert series $h(z)$, the lex-segment plays a special role. We denote it by $L(h)$ or just $L$ if $h(z)$ is clear from the context. 
If $\chara k=0$, then $V(L)$ is dense in $\mathbb{G}(h)$ so that $\dim V(L)=\dim \mathbb{G}(h)$. 
Therefore, according to Corollary \ref{main}, we have 
\begin{equation}
\dim \mathbb{G}(h)=\#\mathcal{S}(L)
\label{brutta} 
\end{equation} 

To double-check, the suspicious reader can show directly that the right-hand side of the formulas \ref{bella} and \ref{brutta} indeed coincide. It is a simple, but not obvious, exercise. 
 
We come now to study the Betti strata of $V(E)$. For a homogeneous ideal $I$ in $k[x,y]$ denote by $\beta_{i,j}(I)$ the $(i,j)$-th Betti number. In particular, $\beta_{0,j}(I)$ is the number of minimal generators of $I$ of degree $j$. It is well known, each pairs of the three sets of invariants $\{\beta_{0,j}(I)\}_j$, $\{\beta_{1,j}(I)\}_j$ and the $\{\dim I_j\}_j$ determine the third. 
Given integers $j$ and $u$ we define: 

 $$V(E, j, u)= \{ I\in V(E) : \beta_{0,j}(I)=u \}$$ 
 $$V(E, j, \geq u)= \{ I\in V(E) : \beta_{0,j}(I) \geq u \}$$ 

If $\beta=(\beta_1,\dots, \beta_j,\dots)$ is a vector with integral entries we define 
$$V(E, \beta)=\bigcap_j V(E, j, \beta_j) $$ 
and 
\begin{equation}
V(E,\geq \beta)=\bigcap_j V(E, j, \geq \beta_j) 
\label{transver}
\end{equation}

We consider a monomial ideal $E$ and its associated sequence $m_0,\dots,m_t$. The ideals in $V(E)$ are parametrized by the affine space $\AA^n$ where $n=\# \mathcal{S}(E)$. We denote by $p_{ij}$ with $(i,j)\in \mathcal{S}(E)$ (or simply by $p_1,\dots,p_n$) the coordinates of $\AA^n$. 
Given $p\in \AA^n$ we consider the matrix $N\in T_3(E)$, defined in (3) of Definition \ref{defi}. Set $M(p)=M_0(E)+N$. By the Hilbert-Burch theorem, the ideal $I$ of maximal minors of $M(p)$ has the free resolution: 

 \begin{equation}
 0\to \bigoplus_{i=1}^tR(-b_i)\overset{M(p)} \to \bigoplus_{i=1}^{t+1} R(-a_i)\to 0
\label{hb} 
 \end{equation}

where $a_i=t+1-i+m_{i-1}$ for $i=1,\dots,t+1$ and $b_i=a_{i+1}+1$ for $i=1,\dots,t$. For every $j$ we set 

$$w_j=\{ i : a_i=j\} \quad \mbox{ and } \quad v_j=\{ i : b_i=j\}.$$
 
Tensoring \ref{hb} with $k$ and taking the degree $j$ component we have the complex of vector spaces
$$k^{\# v_j} \overset{M(p)_j } \to k^{\# w_j}\to 0$$ 
whose homology gives the Betti numbers of $I$. 
Here $M(p)_j$ is the submatrix of $M(p)$ with rows indices $w_j$ and column indices $v_j$. 

It follows that 
\begin{equation}
\beta_{0,j}(I)=\#w_j-\rank M(p)_j
\label{etabeta}
\end{equation} 
and hence $V(E,j, \geq u)$ is the determinantal variety defined by the condition 
 
 $$\rank M(p)_j\leq \# w_j-u.$$ 

If $i_1\in w_j$ and $i_2\in v_j$ then $(i_1,i_2)$-th entry of $M(p)$ is: 

$$
\begin{array}{ll}
p_{i_1i_2} & \mbox{ if } i_1>i_2 \mbox{ and } d_{i_2}>0,\\
0  & \mbox{ if } i_1>i_2 \mbox{ and } d_{i_2}=0,\\
1  & \mbox{ if } i_1=i_2, \\ 
0  & \mbox{ if } i_1<i_2. \\ 
\end{array}
$$
Hence the matrices $M(p)_j$ have entries that are either variables or $0$ or $1$. Furthermore the sets of the variables involved in $M(p)_j$ and in $M(p)_i$ are disjoint if $i\neq j$. 
To summarize: 

\begin{lemma}
\label{trasva}
The variety $V(E,\geq\beta)$ is the transversal intersection of the determinantal varieties $V(E,j, \geq \beta_j)$. 
In particular, the codimension of $V(E,\geq \beta)$ is the sum of the codimensions of the $V(E,j, \geq \beta_j)$ and 
$V(E,\geq \beta)$ is irreducible iff $V(E,j, \geq \beta_j)$ is irreducible for every $j$. 
\end{lemma} 

From now on we concentrate our attention on the variety $V(E,j,\geq u)$. 
 If $i\in w_j\cap v_j$ then $(i,i)$-entry of $M(p)_j$ is $1$ and all the other entries in that column are $0$. So we can simply get rid of the column and the row containing the $1$'s. Denote by ${M(p)_j^*}$ the submatrix that we get from $M(p)_j$ removing the $1$'s together with their columns and rows. Since the $1$'s are in different rows and columns we have 
$$\rank M(p)_j=\rank M(p)_j^*+\# (w_j\cap v_j)$$

Noticing that $ \# (w_j\setminus w_j\cap v_j)$ is exactly $\beta_{0, j}(E)$ we can conclude that:

\begin{lemma} 
 The variety $V(E,j, \geq u)$ is defined by the condition 
$$\rank M(p)_j^*\leq \beta_{0, j}(E) -u$$
\end{lemma} 

The matrices $M(p)_j^*$ have entries that are either $0$ or distinct variables and if the $(i_1,i_2)$-th entry is $0$ the same is true also for the $(h_1,h_2)$-th with $h_1\leq i_1$ and $h_2\geq i_2$, that is, they look like

\begin{equation}
 \left(
 \begin{array}{cccccc} 
\bullet & \bullet & 0 & 0 & 0 & 0 \\ 
\bullet & \bullet & 0 & 0 & 0 & 0 \\ 
\bullet & \bullet & \bullet & 0 & 0 & 0 \\ 
\bullet & \bullet & \bullet & 0 & 0 & 0 \\ 
\bullet & \bullet & \bullet & \bullet &\bullet & 0 \\ 
\bullet & \bullet & \bullet & \bullet & \bullet & \bullet \\ 
\bullet & \bullet & \bullet & \bullet & \bullet & \bullet
\end{array}\right) 
\label{upper} 
\end{equation}

where each $\bullet$ is a distinct variable. 

\begin{remark} The ideals of minors of a given size of the matrices of type \ref{upper} are radical (to prove it one can use Gr\"obner bases) but obviously not prime in general. They can have clearly minimal primes of different codimension. 
\end{remark} 

The following example shows that $V(E,\geq u)$ is not irreducible in general. 
\begin{example} 
Let $E=(x^6,x^5y, x^4y^3, x^3y^4, x^2y^4,xy^5, y^7)$. Here $d=(1,2,1,0, 1,2)$ and 
$a=(6,6,7,7,6,6,7)$ and $b=(7,8,8,7,7,8)$. 
We have that $V(E)$ is an $8$-dimensional affine space parametrized by the matrix
$$
M(p)=\begin{array}{rrrrrrrrr}
& \vline & 7 & 8& 8& 7& 7& 8 \\
\hline 
6&\vline &y & 0& 0& 0& 0& 0 \\
6&\vline&-x & y^2& 0& 0& 0& 0 \\
7&\vline&p_1& -x+p_4y& y& 0& 0& 0 \\
7&\vline&p_2& p_5y& -x& 1& 0& 0 \\
6&\vline&0 & 0& 0& -x& y& 0 \\
6&\vline&0 & 0& 0& 0& -x& y^2 \\
7&\vline&p_3& p_6y& 0& 0& p_7& -x+p_8y
 \end{array}$$ 
 The numbers on the boundary are the degree of the syzygies (the first row) and degree of the generators (the first column). 
 
Here the only interesting variety is $V(E,7,\geq u)$. We have $w_7=\{3,4,7\}$ and $v_7=\{1,4,5\}$. The matrix $M(p)_7$ is obtained by $M(p)$ by selecting the rows and columns marked with $7$:
$$
M(p)_7=
\begin{array}{cccccccc}
&\vline & 7 & 7& 7 \\
\hline 
7&\vline&p_1& 0& 0 \\
7&\vline&p_2& 1& 0 \\
7&\vline&p_3& 0& p_7 
 \end{array}$$ 
To get $M(p)^*_7$ we have to cancel rows and columns containing $1$'s: 
$$
M(p)^*_7=\left(\begin{array}{cccccccc}
 p_1& 0 \\
 p_3& p_7 
 \end{array}
\right) $$ 
 Hence $V(E,7,\geq u)$ is defined by the condition 
 $$\rank M(p)_7^*\leq 2-u$$
Therefore $V(E,7,\geq 1)$ is defined by $p_1p_7=0$ and has  two irreducible components of codimension $1$. The veriety  $V(E,7,\geq 2)$ is defined by $p_1=p_3=p_7=0$ and is irredicible of codimension $3$. 
\end{example}

The above example can be generalize to show that every matrix of type \ref{upper} can arise as $M(p)_j^*$ for some $E$ and some $j$. Instead of given complicated and cumbersome details, we just give an example (hopefully illuminating) leaving the details to the interested readers. 

\begin{example} 
Starting with $E$ associated to the sequence 
$$d=(1, 1, 2, 1, 0, 1, 1, 1, 2, 1, 1, 0, 1, 1, 2, 1, 1, 1)$$
 the matrix 
$M(p)_{19}^*$ is: 
 
 $$
\begin{array}{ccccccc}
 \bullet & \bullet & 0 & 0 & 0 & 0 & 0 \\
 \bullet & \bullet & \bullet & \bullet & \bullet & 0 & 0 \\
 \bullet & \bullet & \bullet & \bullet & \bullet & 0 & 0 \\
 \bullet & \bullet & \bullet & \bullet & \bullet & \bullet & \bullet \\ 
 \bullet & \bullet & \bullet & \bullet & \bullet & \bullet & \bullet \\ 
 \bullet & \bullet & \bullet & \bullet & \bullet & \bullet & \bullet \\
 \bullet & \bullet & \bullet & \bullet & \bullet & \bullet & \bullet 
 \end{array}
 $$
 and $V(E,19,\geq u)$ is defined by condition $\rank M(p)_{19}^*\leq 7-u$. 
\end{example} 
 
If $E$ is a lex-segment then $d_i>0$ for every $i$. This has the effect that the matrices $M(p)_j$ are matrices of indeterminates. Hence we obtain the following results of Iarrobino \cite{I2}: 

\begin{corollary} 
Let $L$ be a lex-segment. 
Then variety $V(L,j,\geq u)$ is defined by the condition $\rank M(p)_j\leq \beta_{0,j}(L)-u$ where $M(p)_j$ is a matrix of distinct variables of size $\beta_{0,j}(L)\times \beta_{1,j}(L)$. In particular: 
\begin{itemize} 
\item[(1)] $V(L,j,\geq u)$ is irreducible. It coincides with the closure of $V(L,j,u)$ provided $V(L,j,u)$ it is not empty, that is, provided $\beta_{0,j}(L)-\beta_{1,j}(L)\leq u \leq \beta_{0,j}(L)$. 
\item[(2)] If $\beta_{0,j}(L)-\beta_{1,j}(L)\leq u \leq \beta_{0,j}(L)$ then the codimension of $V(L,j,\geq u)$ is 
$(\beta_{1,j}(L)-\beta_{0,j}(L)+u)u$. 
\end{itemize}
\end{corollary} 

If $I$ is an ideal with the same Hilbert function of the lex-segment $L$ and $\beta_{0j}(I)=u$ then $\beta_{1,j}(L)-\beta_{0,j}(L)+u$ is exactly $\beta_{1,j}(I)$. Hence the formula for the codimension of $V(L,j,\geq u)$ can be writen as 
 $\beta_{1,j}(I)\beta_{0,j}(I)$. 
 
It follows that:

\begin{corollary} 
Let $L$ be a lex-segment ideal and $I$ a homogeneous ideal with the Hilbert function of $L$. Set $\beta=\{\beta_{0,j}(I)\}$. Then the variety $V(L,\geq \beta)$ is irreducible, it is the closure of $V(L, \beta)$ and it has codimension $\sum_j \beta_{1,j}(I)\beta_{0,j}(I)$. 
\end{corollary}

We conclude the paper with an example. 

\begin{example} 
Let $L=(x^8, x^7y, x^6y^2, x^5y^4, x^4y^5, x^3y^6, x^2y^7, xy^9, y^{10})$. Then $V(L)$ is $\AA^{22}$, the parametrization given via the matrix $M(p)$ 
 
 $$ 
 \begin{array}{cccccccccccc}
 & \vline & 9 & 9 & 10 & 10 & 10 & 10 & 11 & 11 \\
\hline 
8& \vline & y & 0 & 0 & 0 & 0 & 0 & 0 & 0 \\ 
8&\vline& -x & y & 0 & 0 & 0 & 0 & 0 & 0 \\ 
8&\vline& 0 & -x & y^2 & 0 & 0 & 0 & 0 & 0 \\ 
9&\vline& p_1 & p_5 & -x + yp_9& y & 0 & 0 & 0 & 0 \\ 
9&\vline& p_2 & p_6 & yp_{10}& -x & y & 0 & 0 & 0 \\ 
9&\vline& p_3 & p_7 & yp_{11} & 0 & -x & y & 0 & 0 \\ 
9&\vline& p_4 & p_8 & yp_{12}& 0 & 0 & -x & y^2 & 0 \\ 
10&\vline& 0 & 0 & p_{13} & p_{15} & p_{17} & p_{19} & -x + yp_{21} & y \\ 
10&\vline& 0 & 0 & p_{14} & p_{16} & p_{18} & p_{20} & yp_{22} & -x
\end{array}
$$

The matrices whose ranks describe the Betti strata are: 

$$
M(p)_9= \left(
 \begin{array}{cc} 
 p_1 & p_5\\
 p_2 & p_6 \\
 p_3 & p_7\\
 p_4 & p_8 
 \end{array}\right)
 \qquad \mbox{ and} \qquad 
M(p)_{10} = \left(
 \begin{array}{cccc} 
 p_{13 } &p_{15} &p_{17} &p_{19 }\\ 
 p_{14} &p_{16} &p_{18} &p_{20} 
 \end{array}\right)
$$
For instance, with $\beta=(\beta_j)$ defined by $\beta_9=3, \beta_{10}=1$ and $\beta_j=\beta_j(L)$ for $j\neq 9,10$ the Betti strata 
$V(L,\geq \beta)$ is describe by $\rank M(p)_9\leq 1$ and $\rank M(p)_{10}\leq 1$. \end{example}

\end{document}